\documentclass{amsart}
\usepackage[T1]{fontenc}
\usepackage{amsmath,amssymb,amsthm,amsfonts,bm,changepage,cite,float,mathtools}
\usepackage{caption}
\usepackage{outlines}
\usepackage{diagbox}
\usepackage{hyperref}
\usepackage[none]{hyphenat}

\DeclareRobustCommand{\&}{%
  \ifdim\fontdimen1\font>0pt
    \textsl{\symbol{`\&}}%
  \else
    \symbol{`\&}%
  \fi
}

\begin{document}

\begin{center}
{\Large \bf EUCLIDEAN TOURS IN FAIRY CHESS}
\vspace{8mm}

{\large Gabriele Di Pietro and Marco Rip\`a}
\end{center}
\vspace{14mm}

\begin{adjustwidth}{1.5 cm}{1.5 cm}
{\bf Abstract.} The present paper aims to extend the knight's tour problem for $k$-dimensional grids of the form $\{0,1\}^k$ to other fairy chess leapers. Accordingly, we constructively show the existence of closed tours in $2 \times 2 \times \cdots \times 2$ ($k$ times) chessboards concerning the wazir, the threeleaper, and the zebra, for all $k \geq 15$. Our result considers the three above-mentioned leapers and replicates for each of them the recent discovery of Euclidean knight's tours for the same set of $2 \times 2 \times \cdots \times 2$ grids, opening a new research path on the topic by studying different fairy chess leapers that perform jumps of fixed Euclidean length on given regular grids, visiting all their vertices exactly once before coming back to the starting one.
\end{adjustwidth}
\vspace{6mm}

\section{Introduction} 
The famous knight's tour problem \cite{4} asks to perform, on a given chessboard, a sequence of moves of the \emph{knight}, the trickiest chess piece, that visits each square exactly once.

The earliest known reference to the knight's tour problem dates back to the 9th century AD. In Rudrata's Kavyalankara (see Reference \cite{5}) the pattern of a knight's tour has been presented as an elaborate poetic figure. The poet and philosopher Vedanta Desika in his poem Paduka Sahasram (14th century) has composed two Sanskrit verses where the second one can be derived from the first by performing a knight's tour in a $4 \times 8$ board. One of the first known mathematicians who investigated the knight's tour problem was Leonhard Euler \cite{11}. Later, in 1823, van Warnsdorff described the first procedure to complete a knight's tour. At the present time, many other scientific papers have been published around this puzzle and its variation \cite{6,7}. Moreover, if it is possible to reach the starting square with an additional knight move after the last one of a valid knight's tour, the resulting path is closed (Hamiltonian tour); otherwise, the knight's tour is open as each square has been visited exactly once but the beginning square is not reachable at the end.

On February 2024, the paper ``Proving the existence of Euclidean knight's tours on $n \times n \times \cdots \times n$ chessboards for $n < 4$" \cite{1} examined the knight's tour problem for some $k$-dimensional grids
\begin{equation}\label{kchessboard}
 C(n,k) \coloneqq \{0,1,\dots, n-1\}^k,   
\end{equation}
 where $k \in \mathbb{N}^+$ and 
\begin{equation}
\{0,1,\dots, n-1\}^k= \underbrace{\{\{0,1,\dots, n-1\} \times \{0,1,\dots, n-1\} \times \dots \times\{0,1,\dots, n-1\}\},}_{\text{$k$ times}}
\end{equation}
providing the necessary and sufficient condition for the existence of a closed knight's tour on any $C(2,k)$.

Trivially, $|C(n,k)| = n^k$ and so, given $i \in \{0,1,\dots,n^k\}$, the vertex $V_i \in C(n,k)$ is identified by the $k$-tuple of Cartesian coordinates $(x_1,x_2, \dots,x_k)$ : $x_1,x_2, \dots,x_k \in \{0,1,\dots, n-1\}$.
Thereby, on the given grid, each knight's move takes place by moving the piece from one vertex to another. Then, it is natural to associate a \emph{Euclidean} knight's tour to a proper sequence of all the elements of $\{V_1, V_2, \dots, V_{n^k}\}$ so that the Euclidean distance between any two consecutive vertices, $V_i$ and $V_{j:=i+1}$, remains the same by construction.

It is worth pointing out that the FIDE Handbook (see Reference \cite{2}) uses the superlative of \emph{near} as a criterion to state the official knight move rule. Consequently, it is common sense to assume also that, on the chessboard, any move that covers a distance of (exactly) $\sqrt{2^2+1^2}= \sqrt{5}$ units between the centers of the starting and ending square is a knight move. Thus, a knight's \emph{jump}, mathematically speaking, is the connection between two vertices, belonging to the grid $C(n,k)$, which are at a Euclidean distance of $\sqrt{5}$.

\sloppy Accordingly, let us define the distance between the two vertices $V_i=(x_1,x_2,\dots, x_k) $ and $V_j=(y_1,y_2,\dots, y_k)$ though this Euclidean distance $\| V_i-V_j \|: C(n,k) \rightarrow \mathbb{R} $,
\begin{equation}\label{norm}
\| V_i-V_j \| = \sqrt{(x_1-y_1)^2+(x_2-y_2)^2+\cdots+(x_k-y_k)^2}.
\end{equation} 
More specifically, for the \emph{Euclidean $k$-knight}, the distance $\| V_i-V_j \|$ is equal to $\sqrt{5}$.

At this point, it is useful to introduce the following definition.
\newtheorem{polygonal}{Definition}[section]
\begin{polygonal}\label{polygonal}
 Given a Euclidean tour on $C(n,k)$ and its associated distance $d: C(n,k) \rightarrow \mathbb{R}$, a given polygonal chain $P_C(n,k)$ indicates the ordered sequence of all vertices in $C(n,k)$ associated to a Euclidean tour such that the distance between the first vertex and the last one of the tour is $d$. Conversely, $P_O(n,k)$ indicates a valid Euclidean tour where the distance between the first and last vertex is not equal to $d$.
 \end{polygonal}

Now, \emph{A Guide to Fairy Chess} (see Reference \cite{3}) let us extend the Euclidean chess tour concept to the fascinating fairy chess pieces, listed in Table \ref{Tab1}.
\begin{table}[H]
    \centering
    \begin{tabular}{|c|c|c|c|c||c|c}
\hline
 \backslashbox{$\boldsymbol a$}{$\boldsymbol b $} & $\boldsymbol 0$ & $\boldsymbol 1$ & $\boldsymbol 2$ & $\boldsymbol 3$ & $\boldsymbol 4$ & \dots \\
\hline
$\boldsymbol 0$ &  Zero  (0) & Wazir  (W) & Dabbaba   (D) & Threeleaper  (H) & Fourleaper &  \\
\hline 
$\boldsymbol 1$ & Wazir (W) & Ferz (F) & Knight (N) & Camel (C) & Giraffe &  \\
\hline
$\boldsymbol 2$ & Dabbaba (D) & Knight (N) & Alfil (A) & Zebra (Z) & Stag &  \\
\hline 
$\boldsymbol 3$ & Threeleaper (H) & Camel (C) & Zebra (Z) & Tripper (G) & Antelope &  \\
\hline
\hline
$\boldsymbol 4$ & Fourleaper & Giraffe & Stag & Antelope & Commuter &  \\
\hline
\vdots &  &  &  &  &  &   \\

\end{tabular}
    \caption{Fairy chess' leapers.}
    \label{Tab1}
\end{table}

Those chess pieces are known as the \emph{leapers} since they jump from a square of the chessboard to another one at a given (fixed) distance. For instance, the Euclidean knight is described as $(1,2)$-leaper (or $(2,1)$-leaper); in fact, $\sqrt{2^2+1^2}=\sqrt{1^2+2^2}=\sqrt{5}$ as stated above. Furthermore, from the $(a,b)$ pair in Table \ref{Tab1}, it is clear that the canonical move of every fairy chess piece, in the $k$-dimensional grid $C(n,k)$, is obtained by adding or subtracting $a$ from one of the $k$ Cartesian coordinates of the starting vertex and, simultaneously, adding or subtracting $b$ to another of the remaining $k-1$ elements of the mentioned $k$-tuple. In this paper, we refer to any $(a,b)$ pair in Table \ref{Tab1} as an \emph{$(a,b)$-moving rule}.

\sloppy Thus, assuming $\sqrt{a^2+b^2}$ as our distance criterion, other noncanonical movements are allowed in $C(n,k)$, (i.e., for each $k>4$, the \emph{$(1,1,1,1,1)$-moving rule} is another possible move of the $(2,1)$-leaper called knight since $\sqrt{1^2+1^2+1^2+1^2+1^2}=\sqrt{5}$).  

This definition is justified by Article 3.6 of FIDE Handbook \cite{2} since: ``The knight may move to one of the squares \textbf{nearest} to that on which it stands but not on the same rank, file, or diagonal". 

Let the starting vertex $V_0 \equiv (0,0,0,0,0,0)$ of $C(2,6)$ be given. The $(1,1,1,1,1)$-moving rule is performed by adding or subtracting $1$ to five of the Cartesian coordinates of the starting vertex. Then, we only need to add the $1$ components of the $(1,1,1,1,1)$-moving rule to $(0,0,0,0,0,0)$ in order to reach another vertex of our $2 \times 2 \times 2 \times 2 \times 2 \times 2$ grid, say, $(1,0,1,1,1,1)$ or even $(0,1,1,1,1,1)$.

Notably, the uniqueness of Euclidean fairy chess pieces is intrinsically maintained by its canonical version. For instance, both the $(0,5)$-leaper and the $(3,4)$-leaper shares jumps of $\sqrt{25}=5$, and this means that in a $C(2,26)$ grid their moves are the same, but in $C(6,2)$ they jump on different vertices because they have the $(0, 5)$-moving rule and the $(3,4)$-moving rule, respectively. In detail, starting from $(0,0) \in C(6,2)$, the $(0,5)$-leaper alternatively moves to $(5,0)$ or $(0,5)$, while the $(3,4)$-leaper can only reach the vertex $(3,4)$ or the vertex $(4,3)$. 

Thus, we note that (usually) the fairy chess pieces have multiple options.
This is certainly the case of the $(2,3)$-leaper, the notable \emph{zebra}, for which the given jump length of $\sqrt{3^2+2^2}=\sqrt{13}$ make it possible to perform the $(2,1,1,1,1,1,1,1,1,1)$-moving rule, the $(2,2,1,1,1,1,1)$-moving rule, the $(1,1,1,1,1,1,1,1,1,1,1,1,1)$-moving rule, the $(2,2,2,1)$-moving rule, and the $(3,1,1,1,1)$-moving rule. 

To give an example with the $(2,2,2,1)$-moving rule in $C(2,5)$, starting from $(0,0,0,0,0)$ (as usual), the zebra can alternatively reach $(2,0,1,2,2,0)$ or $(0,2,2,2,0,1)$. Anyway, in the present paper, we only consider Hamiltonian Euclidean tours in $C(2,k)$ and consequently, for instance, only the $(1,1,1,1,1,1,1,1,1,1,1,1,1)$-moving rule is available for a $(2,3)$-leaper in order to move on a $C(2,k)$ grid, for any $k \geq 13$. Obviously, for each $k \leq 13$, it is not possible to cover all the vertices of the $C(2,k)$ grid with the special move that requires sums/subtractions of thirteen addends. 

Hence, we can refine Definition \ref{polygonal} as follows. 
\newtheorem{polygonalbis}[polygonal]{Definition}
\begin{polygonalbis}\label{polygonalbis}
Given a Euclidean tour on $C(n,k)$ and a generic fairy chess leaper $L$ with associated distance $d: C(n,k) \rightarrow \mathbb{R}$, the polygonal chain $P^L_C(n,k)$ indicates the ordered sequence of all vertices in $C(n,k)$ covered by $L$ such that distance between the last vertex and the first one is $d$. Conversely,$P^L_O(n,k)$ indicates a valid Euclidean tour of the leaper $L$ where the distance between the first and last vertex is not equal to $d$. 
\end{polygonalbis}

It is notable, that, in Definition \ref{polygonalbis}, the subscript $C$ is referred to a closed Euclidean tour, and $O$ is referred to an open one. Additionally, the closed path can be named \emph{Hamiltonian tour on $C(n,k)$} for its striking similarity to the \emph{ Hamiltonian cycle} and conversely, if the Euclidean tour is \emph{open}, the \emph{open Euclidean tour} denomination can be used. However, looking at Table \ref{Tab1}, we can replace the apex $L$ with any other leaper character (e.g., for the knight case, we have the polygonal chains $P^N_C(n,k)$ and $P^N_O(n,k)$).

Due to computing power limitations, the present paper only looks for wazir's, threeleaper's, and zebra's Euclidean Hamiltonian tours, and then we only need the $P^W_C(n,k)$, $P^T_C(n,k)$, and $P^Z_C(n,k)$ notations.  

\section{Parity of vertices} 
Subsection 4.1 of ``Metric spaces in chess and international chess pieces graph diameters" (see Reference \cite{8}) distinguishes between \emph{even} and \emph{odd} vertices of $C(n,k)$ as follows: given a vertex $V \equiv (x_1, x_2, \dots, x_k)$ of a $k$-dimensional grid $C(n,k)$ where $x_1, x_2, \dots, x_k \in \mathbb{N}$, assuming also $m \in \mathbb{N}$, it is possible to define $V$ \emph{even} if and only if
\begin{equation}
    \sum_{j=1}^k x_j = 2 m,
\end{equation} 
whereas we define $V$ to be \emph{odd} otherwise.
 
\newtheorem{vertices}{Lemma}[section]
\begin{vertices}\label{vertices}
Let $n, k, x_1, x_2, \dots, x_k \in \mathbb{N}$ and assume that $V \equiv (x_1, x_2, \dots, x_k)$ is an even vertex of the grid $C(n,k)$. Then, the number of odd coordinates of $V$ is even.   
\end{vertices}

\begin{proof}
Let $X:=\{x_1, x_2, \dots, x_k\}$ be the whole set of the $k$ coordinates of $V$, the given vertex of $C(n,k)$. If we indicate as $\{d_1, d_2, \dots, d_s\}$ the set of the odd elements of $X$, and as $\{p_1, p_2, \dots, p_t\}$ the set of the even coordinates of $X$, $s+t=k \in \mathbb{N}$ follows by construction.

Since $V$ is an even vertex by hypothesis, the following equality holds.
\begin{equation}
    \sum_{j=1}^k x_j = \sum_{j=1}^s d_j + \sum_{j=1}^t p_j = 2 m \quad (m \in \mathbb{N})
\end{equation}
and, secondly, $\sum_{j=1}^t p_j= 2 h$ for each $h \in \mathbb{N}$ with $h \leq m$.

Hence, 
\begin{equation}
    \sum_{j=1}^s d_j  = 2 m - 2 h = 2 (m-h).
\end{equation}
Therefore, $s$ is even and this concludes the proof.
\end{proof}
\sloppy Similarly, we can distinguish between even and odd \emph{leaper moves} by invoking the previously stated distance criterion for the fairy chess pieces. In fact, given a leaper and its $(x_1, x_2, \dots, x_{k'})$-moving rule in a $k$-dimensional grid $C(n,k)$, where $k' \leq k$ and $m \in \mathbb{N}$, we define the  $(x_1, x_2, \dots, x_{k'})$-moving rule as \emph{even} if and only if 
\begin{equation}
\sum_{j=1}^{k'} x_j = 2 m,
\end{equation}
otherwise we define the $(x_1, x_2, \dots, x_{k'})$-moving rule to be \emph{odd}.
 
\newtheorem{vertices2}[vertices]{Lemma}
\begin{vertices2}\label{vertices2}
Let $x_1, x_2, \dots, x_k, k', k, n \in \mathbb{N}$, $k' \leq k$, and assume that $\sqrt{x_1^2+x_2^2+ \cdots + x_{k'}^2}$ indicates the Euclidean distance between a pair of vertices of the given $C(n,k)$ grid. The $(x_1,x_2,\dots,x_{k'})$-moving rule is even if and only if $x_1^2+x_2^2+ \dots + x_{k'}^2$ is even, whereas the $(x_1,x_2,\dots,x_{k'})$-moving rule is odd if and only if $x_1^2+x_2^2+ \dots + x_{k'}^2$ is also odd.
\end{vertices2}
\begin{proof}
 If the radicand of $\sqrt{x_1^2+x_2^2+ \dots + x_{k'}^2}$ is even, it is notable that for any $m \in \mathbb{N}$
 \begin{eqnarray*}
    x_1^2+x_2^2+ \dots + x_{k'}^2 = 2 m \Rightarrow \\
     (x_1+x_2+ \dots+ x_{k'})(x_1+x_2+ \dots+ x_{k'})-2 \sum_{i=1}^{k'}\sum_{j=1}^{k'}x_i x_j = 2m \Rightarrow \\
     (x_1+x_2+ \dots + x_{k'})(x_1+x_2+ \dots + x_{k'})= 2m- 2\sum_{i=1}^{k'}\sum_{j=1}^{k'}x_i x_j.
 \end{eqnarray*}
 
In the above, we observe that the right-hand side is even while the left-hand side is the product of $(x_1+x_2+ \dots + x_{k'})$ by itself. Hence, if this product is even, the quantity $(x_1+x_2+ \dots + x_{k'})$ is also even. On the other hand, if we assume the aforementioned product to be odd, it follows that $(x_1+x_2+ \dots + x_{k'})$ should also be odd. Conversely, since the $(x_1+x_2+ \dots + x_{k'})$-moving rule is (alternatively) even or odd, the radicand $x_1^2+x_2^2+ \dots + x_{k'}^2$ can consistently be written by even or odd terms as
 \begin{eqnarray*}
      (x_1+x_2+ \dots+ x_{k'})(x_1+x_2+ \dots+ x_{k'})-2 \sum_{i=1}^{k'}\sum_{j=1}^{k'}x_i x_j ,
 \end{eqnarray*}
and this proves the present lemma. 
\end{proof}
\sloppy Accordingly, considering the leapers included in Table \ref{Tab1} and their possible $(x_1,x_2, \dots, x_{k'})$-moving rules for given $n \times n \times \cdots \times n$ grids $C(n,k)$, Theorem \ref{parita} follows.
\newtheorem{parita}{Theorem}[section]
\begin{parita}\label{parita}
    Let $n,k \in \mathbb{N}-\{0,1\}$ so that the $k$-dimensional grid $C(n,k)$ is given. Then, consider the $(a,b)$-leaper in $C(n,k)$ such that $a+b$ is even. If the $(a,b)$-leaper starts from an even starting vertex, it can only visit (some of) the $\lceil\frac{n^{k}}{2}\rceil-1$ even vertices, otherwise, if the $(a,b)$-leaper starts from an odd starting vertex, it can only visit (some of) the $\lfloor\frac{n^{k}}{2}\rfloor-1$ odd vertices.     
\end{parita}
\begin{proof}
There are $n^k$ vertices in $C(n,k)$. Consequently, the number of even and odd vertices is $\lceil\frac{n^{k}}{2}\rceil$ and $\lfloor\frac{n^{k}}{2}\rfloor$, respectively. Then, we only need to prove that each $(a,b)$-leaper such that $a+b$ is even can only visit even vertices if the piece is initially placed on an even vertex and vice versa. This implies that the maximum cardinality of each set of vertices belonging to any even/odd $(a,b)$-leaper tour which satisfies the above cannot exceed the number of even/odd vertices of $\{0,1,\ldots,n-1\}^k$.

Let us call $d_1, d_2, \ldots, d_s$ the odd coordinates of a given vertex of $C(n,k)$, and conversely let $p_1, p_2, \dots, p_t$ indicate the even coordinates of the same vertex ($s+t=k$ follows by construction). Without loss of generality, assume that the starting vertex, $V_0 \equiv (d_1, d_2, \ldots, d_s, p_1, p_2, \ldots, p_t)$, is even so that, by Lemma $\ref{vertices}$, $s$ (i.e., the number of the odd coordinates of $V_0$) is even.
    
By invoking Lemma $\ref{vertices2}$, it follows that if $a+b$ is even, then the jumping length characterizing our $(a,b)$-leaper is $\sqrt{a^2+b^2}$, which is also an even number. In general, every linear combination $\tilde{d_1}, \tilde{d_2}, \dots, \tilde{d_{s'}}, \tilde{p_1}, \tilde{p_2}, \dots, \tilde{p_{t'}}$ associated to the same distance is even since  
    \begin{equation*}
       a^2+b^2= \tilde{d_1}^2+\tilde{d_2}^2+ \dots+ \tilde{d_{s'}}^2+\tilde{p_1}^2+\tilde{p_2}^2+ \dots+ \tilde{p_{t'}}^2 \textnormal{,}
    \end{equation*}
     where $s',t' \in \mathbb{N}$ and $s'$ is even, while $\tilde{d_1}, \tilde{d_2}, \dots, \tilde{d_{s'}}$ are the odd coordinates and $\tilde{p_1}, \tilde{p_2}, \dots, \tilde{p_{t'}}$ are the even ones.

  We now observe how we can apply the $(\tilde{d_1},\tilde{d_2}, \dots, \tilde{d_{s'}},\tilde{p_1},\tilde{p_2}, \dots, \tilde{p_{t'}})$-moving rule to move a fairy chess piece from its starting spot. Since the even coordinates $\tilde{p_1},\tilde{p_2}, \dots, \tilde{p_{t'}}$ do not affect the parity of the starting vertex, given the fact that $d_1+p = d_2$ and $p_1+p= p_2$ hold for any odd $d_1,d_2 \in \mathbb{N}$ and for every even $p, p_1, p_2 \in \mathbb{N}$, we are free to consider only the $\tilde{d_1}, \tilde{d_2}, \dots, \tilde{d_{s'}}$ coordinates.

At this point, we have to distinguish between three cases, depending on how the odd $\tilde{d_1}, \tilde{d_2}, \dots, \tilde{d_{s'}}$ coordinates are applied to $V_0$.
 
\begin{outline}[enumerate]
\item \1 First of all, let us assume that $\tilde{d_1}, \tilde{d_2}, \dots, \tilde{d_{s'}}$ only change the values of $s'$ elements of the set $\{d_1,d_2, \ldots, d_s\}$ (the odd coordinates of $V_0$) and, in particular, let $s'$ be strictly smaller than $s$. It follows that $s'$ elements of the set $\{d_1,d_2, \dots, d_s\}$ become even. Since $s-s'$ is even, the sum of the remaining $s-s'$ odd coordinates is also even, and, after making the $(a,b)$-leaper move, we have that the reached vertex is even too. \newline
On the other hand, if $s'=s$, all the $d_1,d_2, \dots, d_s$ coordinates of the reached vertex become even and, consequently, the considered fairy chess piece lands on an even vertex.
\item Secondly, let us assume that $\tilde{d_1}, \tilde{d_2}, \dots, \tilde{d_{s'}}$ change only the values of $s'$ coordinates among $p_1,p_2,\dots, p_t$ (i.e., the even coordinates of $V_0$). Since $s'$ is even, $s'$ even elements of $\{p_1,p_2,\dots, p_t\}$ become odd so that their sum is even, and thus the reached vertex is even as well.
\item Lastly, we assume that the mentioned $\tilde{d_1}, \tilde{d_2}, \ldots, \tilde{d_{s'}}$ change the values of a subset of the $\{d_1, d_2, \dots, d_s, p_1, p_2, \dots, p_t\}$ coordinates of $V_0$.
For example, without loss of generality, we are allowed to assume that, for any pair of nonnegative integers $(s_1, s_2)$, $s'_1+s'_2=s'$ so that $s'_1$ odd coordinates of $\{d_1, d_2, \dots, d_s\}$ become even and $s'_2$ even coordinates of $\{p_1, p_2, \dots, p_t\}$ become odd. Since $s'$ is even, we distinguish two subcases: the case where both $s'_1$ and $s'_2$ are even, and the case where $s'_1$ and $s'_2$ are two odd integers.
\2 Let $s'_1 < s$ and $s'_1,s'_2$ be even. It follows that the sum of the remaining $s-s'_1$ odd coordinates of $\{d_1, d_2, \dots, d_s\}$ is even and the sum of $s'_2$ odd coordinates is also even (given the fact that $s'_2$ is even so that the selected fairy chess piece lands on an even vertex of $C(n,k)$). Alternatively, if $s'_1 = s$, it follows that all the $s$ odd coordinates $d_1, d_2, \dots, d_s$ become even, and then the reached vertex is even. 
\item Let $s'_1 < s$ and $s'_1,s'_2$ be odd. We have that the sum of the remaining $s-s'_1$ odd coordinates of $\{d_1, d_2, \dots, d_s\}$ is odd and the sum of $s'_2$ odd coordinates is also odd.
As a result, since the sum of two odd numbers is even, we have that the landing spot of the considered $(a,b)$-leaper is, again, an even vertex. 
\end{outline}
  
   A similar reasoning can be made as the starting vertex $V_0$ is odd, finally proving the theorem.  
\end{proof}
Applying Theorem $\ref{parita}$ to the pieces included in Table $\ref{Tab1}$, we conclude that Hamiltonian fairy chess tours are possible for wazir, threeleaper, knight, giraffe, zebra, antelope, and so forth.

In detail, we know that such knight's tours are always possible in $C(2,k)$ as $k$ becomes sufficiently large \cite{1}, while, considering the same family of grids, the currently available computing power has allowed us to research the wazir's tours, the threeleaper's tours, and even the zebra's ones. 

\section{Hamiltonian tours of fairy chess}

In 2007, Tom\'a\v{s} and Petr proved the existence of Hamiltonian paths in hypercubes \cite{9}. Here we show constructive proof for the wazir's tour.

\newtheorem{wazir}{Theorem}[section]
\begin{wazir}
A Hamiltonian Euclidean wazir's tour $P^W_C(2,k)$ exists for each positive integer $k$.    
\end{wazir}
\begin{proof}
Trivially, $P^W_C(2,1) \coloneqq (0) \rightarrow (1)$ describes a wazir's tour for $C(2,1)$, and we note that this tour is also Hamiltonian (since the Euclidean distance between the vertices $(0)$ and $(1)$ is $ \sqrt{0^2+1^2}=\sqrt{1}=1$). Then, it is possible to move $P^W_C(2,1)$ from $C(2,1)$ to $C(2,2)$ adding a new coordinate at the right-hand side in order to construct $P^W_{C_1}(2,2) \coloneqq(0,0) \rightarrow (1,0)$ and $P^W_{C_2}(2,2) \coloneqq (0,1) \rightarrow (1,1)$.

Hence, by reverting the tour $P^W_{C_2}(2,2)$, we get $\hat{P^W_{C_2}}(2,2) \coloneqq (0,1) \leftarrow (1,1)$ and so, connecting the ending vertex of $P^W_{C_1}(2,2)$ with the starting vertex of $\hat{P^W_{C_2}}(2,2)$, the new wazir's tour $P^W_C(2,2) \coloneqq (0,0) \rightarrow (1,0) \rightarrow (1,1) \rightarrow (0,1)$ is finally constructed. Using the same procedure, we consequently get the wazir's tour $P^W_C(2,3) \coloneqq(0,0,0) \rightarrow (1,0,0) \rightarrow (1,1,0) \rightarrow (0,1,0) \rightarrow (0,1,1) \rightarrow (1,1,1) \rightarrow (1,0,1) \rightarrow (0,0,1)$, and then we can repeat the same process, for each $k >3$.   
\end{proof}
As we explained in the introduction, the currently available computing power has allowed us to research a closed Hamiltonian threeleaper's tour in $C(2,10)$ and a closed Hamiltonian zebra's one in $C(2,14)$, but only Hamiltonian threeleaper's tours for the $C(2,11)$ grid and Hamiltonian zebra's tours for the $C(2,15)$ grid have been found. Due to their length, we have decided to upload on Zenodo the solutions $P^T_C(2,11)$ and $P^Z_C(2,15)$ (choosing the binary representation of the vertices with the aim to enlight the patterns arising from the representation of the given polygonal chains).

For instance, the binary representation of the vertex $(0,0,0,1,0,0,1,0,0,1,1,0,0,1,0) \in C(2,15)$ is $000000000110010$, a number obtained by listing the mentioned coordinates from left to right. 

Hence, about the threeleaper, we have the following result.
\newtheorem{threeleaper}[wazir]{Theorem}
\begin{threeleaper}
A Hamiltonian Euclidean threeleaper's tour $P^T_C(2,k)$ exists for each integer $k~\geq~11$.    
\end{threeleaper}
\begin{proof}
Firstly, only the $(1,1,1,1,1,1,1,1,1)$-moving rule can be applied to the context of a Euclidean threeleaper in $C(2,k)$, and thus the condition $k \geq 9$ is mandatory in order to perform any threeleaper jump inside the given grid.
    
However, as $k=9$, we observe that the threeleaper cannot visit all the vertices of $C(2,9)$ (e.g., if the starting vertex is $V_0 \equiv (0,0,0,0,0,0,0,0,0)$, then the only reachable vertex is $V_1 \equiv (0+1, 0+1, 0+1, 0+1, 0+1, 0+1, 0+1, 0+1, 0+1)= (1,1,1,1,1,1,1,1,1)$ and now, using again the $(1,1,1,1,1,1,1,1,1)$-moving rule, it is only possible to subtract every $1$ to the coordinates of $V_1$, coming back to $V_0$).

On the other hand, for $k=11$, a Hamiltonian tour is provided by the polygonal chain 
    \begin{eqnarray*}
    P^T_C(2,11) \coloneqq (0,0,0,0,0,0,0,0,0,0,0) \rightarrow \\ (0,0,1,1,1,1,1,1,1,1,1) \rightarrow \cdots \rightarrow (0,1,1,1,1,1,1,0,1,1,1)
    \end{eqnarray*}
    \sloppy described in the data file \url{https://zenodo.org/records/11199717} (DOI: \url{10.5281/zenodo.11199717}); there, the Euclidean distance between the final and the starting vertex is 
    \begin{equation*}
        \| (0,1,1,1,1,1,1,0,1,1,1)-(0,0,0,0,0,0,0,0,0,0,0)\| = \sqrt{9},
    \end{equation*}
and this proves the existence of a Hamiltonian threeleaper's tour in $C(2,11)$.
Then, to extend this solution to $C(2,12)$, it is sufficient to note that $C(2,11)$ is simply the set of the $2^{11}$ corners of a $11$-cube. Thus, since any vertex of an $11$-face belonging to a $12$-cube is connected to some other vertices belonging to the opposite $11$-face of the same $12$-cube by as many minor diagonals, it is possible to take the solution for the $k=11$ case and reproduce it on the opposite $11$-face of the mentioned $12$-cube. Now, it is possible to mirror/rotate the $11$-face in order to connect the endpoints of both the covering paths of the two $11$-faces through as many diagonals of (Euclidean) length $\sqrt{9}$.    
   In detail, we can extend the $k=11$ solution 
    \begin{eqnarray*}
        P^T_C(2,11) = (0,0,0,0,0,0,0,0,0,0,0) \rightarrow \\ (0,0,1,1,1,1,1,1,1,1,1,1,1) \rightarrow  \cdots \rightarrow (0,1,1,1,1,1,1,0,1,1,1)
    \end{eqnarray*}
        to $k=12$ as follows.
    \begin{enumerate}
        \item In order to move $P^T_C(2,11)$ from $C(2,11)$ to $C(2,12)$, we need to duplicate it as 
        \begin{eqnarray*}
        P^T_{C_1}(2,12) \coloneqq(0,0,0,0,0,0,0,0,0,0,0,0) \rightarrow \\(0,0,1,1,1,1,1,1,1,1,1,0) \rightarrow \cdots \rightarrow (0,1,1,1,1,1,1,0,1,1,1,0)
         \end{eqnarray*}
        and 
        \begin{eqnarray*}
        P^T_{C_2}(2,12) \coloneqq(0,0,0,0,0,0,0,0,0,0,0,1) \rightarrow \\(0,0,1,1,1,1,1,1,1,1,1,1) \rightarrow \cdots \rightarrow (0,1,1,1,1,1,1,0,1,1,1,1),
        \end{eqnarray*}
        adding a new coordinate at the right-hand side.
        \item Now we have to mirror/rotate the $11$-face joined by the polygonal chain $P^T_{C_2}(2,12)$; to achieve this goal, it is sufficient starting from the left-hand side, switching the first $9-1$ coordinates of $P^T_{C_2}(2,12)$, and finally obtaining the new polygonal chain 
        \begin{eqnarray*}
            \tilde{P^T_{C_2}}(2,12) \coloneqq(1,1,1,1,1,1,1,1,0,0,0,1) \rightarrow \\ (1,1,0,0,0,0,0,0,1,1,1,1) \rightarrow \cdots \rightarrow (1,0,0,0,0,0,0,1,1,1,1,1).
        \end{eqnarray*}
        \item Naturally, $\tilde{P^T_{C_2}}(2,12)$ is a Hamiltonian path because the Euclidean distance between the last and the first vertex is $\sqrt{9}$, as the distance between any two consecutive vertices of the given polygonal chain.
        \item Finally, we can connect the $11$-face of the $12$-cube to the opposite $11$-face by considering the reverse path of $\tilde{P^T_{C_2}}(2,12)$, which is defined by
        \begin{eqnarray*}
            \hat{P^T_{C_2}}(2,16) \coloneqq(1,1,1,1,1,1,1,1,0,0,0,1) \leftarrow \\ (1,1,0,0,0,0,0,0,1,1,1,1) \leftarrow \cdots \leftarrow (1,0,0,0,0,0,0,1,1,1,1,1).
        \end{eqnarray*}
        This is correct since the polygonal chain 
        \begin{eqnarray*}
            P^T_C(2,12) \coloneqq(0,0,0,0,0,0,0,0,0,0,0,0) \rightarrow \\(0,0,1,1,1,1,1,1,1,1,1,0) \rightarrow \cdots \rightarrow (0,1,1,1,1,1,1,0,1,1,1,0) \rightarrow \\(1,0,0,0,0,0,0,1,1,1,1,1) \rightarrow   \cdots  \rightarrow (1,1,0,0,0,0,0,0,1,1,1,1) \rightarrow \\(1,1,1,1,1,1,1,1,0,0,0,1)
        \end{eqnarray*} is obtained by connecting the ending point of $P^T_{C_1}(2,12)$ to the starting point of $\hat{P^T_{C_2}}(2,12)$. 
        \end{enumerate}
Consequently, $P^T_C(2,12)$ is a threeleaper Hamiltonian tour since the Euclidean distance between the starting vertex $(0,0,0,0,0,0,0,0,0,0,0,0)$ and the ending vertex $(1,1,1,1,1,1,1,1,0,0,0,1)$ is $\sqrt{9}$ while the polygonal chains $P^T_{C_1}(2,12)$ and $\hat{P^T_{C_2}}(2,12)$ are Hamiltonian by construction.

Then, the described process can be iterated to extend the $12$-cube solution to the $13$-cube, and so forth.

\sloppy Therefore, for each $C(2,k)$ grid such that $k \geq 11$, we have shown the existence of a Hamiltonian threeleaper's tour, and this concludes the proof.   
\end{proof}

With regard to the zebra, we can prove a similar result.
\newtheorem{zebra}[wazir]{Theorem}
\begin{zebra}
A Hamiltonian Euclidean zebra's tour $P^Z_C(2,k)$ exists for each integer $k \geq 15$. 
\end{zebra}
\begin{proof}
Firstly, only the $(1,1,1,1,1,1,1,1,1,1,1,1,1)$-moving rule can be applied to the context of a Euclidean zebra in $C(2,k)$, and thus the condition $k \geq 13$ is mandatory in order to perform any zebra jumps inside the given grid.
    
    But then again (as for the case $k=9$ with reference to the threeleaper tour), as $k=13$ is given, we should note that the zebra cannot visit all the vertices of $C(2,13)$ (e.g., if the starting vertex is $V_0 \equiv (0,0,0,0,0,0,0,0,0,0,0,0,0)$, then the only reachable vertex is $V_1 \equiv (0+1, 0+1, 0+1, 0+1, 0+1, 0+1, 0+1, 0+1, 0+1, 0+1, 0+1, 0+1, 0+1)= (1,1,1,1,1,1,1,1,1,1,1,1,1)$ and now, applying the $(1,1,1,1,1,1,1,1,1,1,1,1,1)$-moving rule once more, it is only possible to subtract every $1$ to the coordinates of $V_1$, coming back to $V_0$).
    
    On the other hand, for the $k=15$ case, a Hamiltonian tour is provided by the polygonal chain
    \begin{eqnarray*}
    P^Z_C(2,15) \coloneqq(0,0,0,0,0,0,0,0,0,0,0,0,0,0,0) \rightarrow \\
    (0,0,1,1,1,1,1,1,1,1,1,1,1,1,1) \rightarrow \dots \rightarrow (0,1,1,1,1,1,1,1,1,1,1,1,1,1,0)
    \end{eqnarray*}
    \sloppy described in the data file \url{https://zenodo.org/records/11490687} (DOI: \url{10.5281/zenodo.11490687}); there, the Euclidean distance between the final and the starting vertex is 
    \begin{equation*}
        \| (0,1,1,1,1,1,1,1,1,1,1,1,1,1,0)-(0,0,0,0,0,0,0,0,0,0,0,0,0,0,0)\| = \sqrt{13},
    \end{equation*}
     and this proves the existence of a Hamiltonian zebra's tour in $C(2,15)$. Then, to extend this solution to $C(2,16)$, it is sufficient to observe that $C(2,15)$ is the set of the $2^{15}$ corners of a $15$-cube. Thus, since any vertex of a $15$-face belonging to a $16$-cube is connected to some other vertices belonging to the opposite $15$-face of the same $16$-cube by as many minor diagonals, it is possible to take the solution for the $k=15$ case and reproduce it on the opposite $15$-face of the mentioned $16$-cube. Now, it is possible to mirror/rotate the $15$-face in order to connect the endpoints of both the covering paths of the two $15$-faces through as many diagonals of (Euclidean) length $\sqrt{13}$.    
    In detail, we can extend the $k=15$ solution 
    \begin{eqnarray*}
        P^Z_C(2,15) = (0,0,0,0,0,0,0,0,0,0,0,0,0,0,0) \rightarrow \\ (0,0,1,1,1,1,1,1,1,1,1,1,1,1,1) \rightarrow  \cdots \rightarrow (0,1,1,1,1,1,1,1,1,1,1,1,1,1,0)
    \end{eqnarray*}
        to $k=16$ as follows.
    \begin{enumerate}
        \item In order to move $P^Z_C(2,15)$ from $C(2,15)$ to $C(2,16)$, we duplicate it as 
        \begin{eqnarray*}
        P^Z_{C_1}(2,16) \coloneqq(0,0,0,0,0,0,0,0,0,0,0,0,0,0,0,0) \rightarrow \\(0,0,1,1,1,1,1,1,1,1,1,1,1,1,1,0) \rightarrow \cdots \rightarrow (0,1,1,1,1,1,1,1,1,1,1,1,1,1,0,0)
         \end{eqnarray*}
        and
        \begin{eqnarray*}
        P^Z_{C_2}(2,16) \coloneqq(0,0,0,0,0,0,0,0,0,0,0,0,0,0,0,1) \rightarrow \\(0,0,1,1,1,1,1,1,1,1,1,1,1,1,1,1) \rightarrow \cdots \rightarrow (0,1,1,1,1,1,1,1,1,1,1,1,1,1,0,1)
        \end{eqnarray*}
        (by adding a new coordinate at the right-hand side, as usual).
        \item Now we have to mirror/rotate the $15$-face joined by the polygonal chain $P^Z_{C_2}(2,16)$; for this purpose, it is sufficient starting from the left-hand side, switching the first $13-1$ coordinates of $P^Z_{C_2}(2,16)$, and finally getting the new polygonal chain 
        \begin{eqnarray*}
            \tilde{P^Z_{C_2}}(2,16) \coloneqq(1,1,1,1,1,1,1,1,1,1,1,1,1,0,0,1) \rightarrow \\(1,1,0,0,0,0,0,0,0,0,0,0,1,1,1,1) \rightarrow \cdots \rightarrow (1,0,0,0,0,0,0,0,0,0,0,0,1,1,0,1).
        \end{eqnarray*}
        \item Naturally, $\tilde{P^Z_{C_2}}(2,16)$ is a Hamiltonian path because the Euclidean distance between the last and the first vertex is $\sqrt{13}$, as the distance between any two consecutive vertices of the given polygonal chain.
        \item Finally, we can connect the $15$-face of the $16$-cube to the opposite $15$-face by considering the reverse path of $\tilde{P^Z_{C_2}}(2,16)$, which is defined by
        \begin{eqnarray*}
            \hat{P^Z_{C_2}}(2,16) \coloneqq(1,1,1,1,1,1,1,1,1,1,1,1,0,0,0,1) \leftarrow \\(1,1,0,0,0,0,0,0,0,0,0,0,1,1,1,1) \leftarrow \cdots \leftarrow (1,0,0,0,0,0,0,0,0,0,0,0,1,1,0,1).
        \end{eqnarray*}
        This is correct since the polygonal chain 
        \begin{eqnarray*}
            P^Z_C(2,16) \coloneqq(0,0,0,0,0,0,0,0,0,0,0,0,0,0,0,0) \rightarrow \\(0,0,1,1,1,1,1,1,1,1,1,1,1,1,1,0) \rightarrow  \cdots  \rightarrow (0,1,1,1,1,1,1,1,1,1,1,1,1,1,0,0) \rightarrow \\(1,0,0,0,0,0,0,0,0,0,0,0,1,1,0,1) \rightarrow   \cdots  \rightarrow (1,1,0,0,0,0,0,0,0,0,0,0,1,1,1,1) \rightarrow \\(1,1,1,1,1,1,1,1,1,1,1,1,0,0,0,1)
        \end{eqnarray*} is obtained by connecting the ending point of $P^Z_{C_1}(2,16)$ to the starting point of $\hat{P^Z_{C_2}}(2,16)$. 
        \end{enumerate}
        Consequently, $P^Z_C(2,16)$ is a zebra Hamiltonian tour since the Euclidean distance between the starting vertex $(0,0,0,0,0,0,0,0,0,0,0,0,0,0,0,0)$ and the ending vertex $(1,1,1,1,1,1,1,1,1,1,1,1,0,0,0,1)$ is $\sqrt{13}$ while the polygonal chains $P^Z_{C_1}(2,16)$ and $\hat{P^Z_{C_2}}(2,16)$ are Hamiltonian by construction.
        
        Then, the described process can be iterated to extend the $16$-cube solution to the $17$-cube, and so forth. 
        
        Therefore, for each $C(2,k)$ grid such that $k \geq 15$, we have shown the existence of a Hamiltonian zebra's tour, and this proves the present theorem.
    \end{proof}
\sloppy The algorithm used here to extend $P^T_C(2,11)$ and $P^Z_C(2,15)$ to the $12$-cube and $16$-cube (respectively) is the same that has been described in the paper \emph{``Proving the existence of Euclidean knight's tours on $n \times n \times \cdots \times n$ chessboards for $n < 4$"} \cite{1}, so we lastly point out that the above method can also be recycled for others fairy chess pieces as we aim to extend their known Hamiltonian tours in $C(2,k)$ to higher dimensions. 

\section{Conclusion}
With regards to $C(2,k)$, every entry of the sub-matrix underlined in Table $\ref{Tab1}$ has been investigated since Theorem $\ref{parita}$ excludes all fairy chess leapers but wazir, threeleaper, knight, and zebra (given the fact that Reference \cite{1} constructively proves the existence of Hamiltonian Euclidean knight's tours on infinitely many grids $C(2,k)$ while the present paper achieves the same result for the other three mentioned leapers).

Actually, we have only proven the existence of Hamiltonian Euclidean threeleaper's and zebra's tours in $C(2,k)$ under the assumptions that $k \geq 11$ and $k \geq 15$, respectively. Thus, the problem 
of proving or disproving the existence of Hamiltonian Euclidean tours is entirely open for the threeleaper in $C(2,10)$ and the zebra in $C(2,14)$.

Although the current calculating power does not allow us to extend our analysis to different fairy chess leapers, it would be interesting to examine the existence of Hamiltonian Euclidean tours in $C(3,k)$ for sufficiently large integers $k$.

\section*{Acknowledgments}
We are very grateful to Aldo Roberto Pessolano for the algorithm in Python code used to discover $P^T_C(2,11)$ and $P^Z_C(2,15)$. Indeed, we acknowledge Tony Di Febo, sincerely thanking him for sharing the personal computer on which we performed the calculations necessary to find the listed Hamiltonian Euclidean tours.
    
\makeatletter
\renewcommand{\@biblabel}[1]{[#1]\hfill}
\makeatother

\section*{Appendix}

The following script is the \emph{Python code} used to study the threeleaper and zebra Hamiltonian closed tours. The following is a brute force algorithm and the code has been running on a QuadCore Intel Core $i7-2600$, $3700$ Mhz while the operative system has been Microsoft Windows 8.1 Professional.

The polygonal chain $P^T_C(2,11)$ was found in about three seconds while we spent about thirty seconds to find the polygonal chain $P^Z_C(2,15)$.  

\begin{verbatim}
    def search(T, k, n, casi):
    history = []
    fullHistory = []
    backtrack = False
    steps = [sum(2**i for i in subset) for subset
    in subsets(range(0, k), n)]
    crash = 0
    quit = 0
    solutions = []

    history.append(T)
    while len(history) < 2**k + 2 and crash < 10**12:
        crash += 1
        if crash % 100000 == 0:
            print(f"First {crash} cases verified.
            Verifying: {history}")

        if len(history) == 2**k + 1 and history[-1] == T:
            quit += 1
            if quit <= casi:
                solution = '\n'.join([bin(num)[2:].zfill(k)
                for num in history])
                solutions.append(solution)
                print(f"Found Hamilton cycle {quit}:\n{solution}")
                if quit == casi:
                    with open("hamilton_cycles.txt", "w") as file:
                        file.write("Hamilton Cycles:\n\n")
                        file.write('\n\n'.join(solutions))
                    return
            else:
                history.pop()
                if backtrack:
                    history.pop()
                    backtrack = False
        else:
            if history[-1] == T and len(history) != 1:
                history.pop()
            if backtrack:
                history.pop()
                backtrack = False

        for i in range(len(steps)):
            if i == len(steps) - 1:
                backtrack = True
            step = steps[i]
            nextT = history[-1] ^ step
            if nextT not in history or nextT == T:
                history.append(nextT)
                if history not in fullHistory:
                    fullHistory.append(history.copy())
                else:
                    history.pop()
                    continue
                break

def subsets(iterable, r):
    pool = tuple(iterable)
    n = len(pool)
    if r > n:
        return
    indices = list(range(r))
    yield tuple(pool[i] for i in indices)
    while True:
        for i in reversed(range(r)):
            if indices[i] != i + n - r:
                break
        else:
            return
        indices[i] += 1
        for j in range(i+1, r):
            indices[j] = indices[j-1] + 1
        yield tuple(pool[i] for i in indices)

import time

def main():
    k = int(input("Number of dimensions (int): "))
    n = int(input("Hamming distance (int): "))
    casi = int(input("Number of solutions to find (int): "))
    T = 0

    start_time = time.time()
    search(T, k, n, casi)
    end_time = time.time()

    execution_time = end_time - start_time
    print(f"\nExecution time: {execution_time:.5f} seconds")

if __name__ == "__main__":
    main()
\end{verbatim}
\end{document}